\documentclass[a4paper]{amsart}
\usepackage{amssymb}

\input xy
\xyoption{all}

\newcommand{\Id}{\mathrm{Id}}
\newcommand{\Comp}{\mathsf{Comp}}

\newcommand{\pr}{\mathrm{pr}}

\newcommand{\id}{\mathrm{id}}

\newcommand{\A}{\mathcal A}

\newcommand{\C}{\mathcal C}
\newcommand{\D}{\mathcal D}

\newcommand{\R}{\mathbb R}
\newcommand{\B}{\mathbb B}
\newcommand{\T}{\mathbb T}
\newcommand{\F}{\mathbb F}
\newcommand{\M}{\mathbb M}

\newtheorem{theorem}{Theorem}
\newtheorem{lemma}{Lemma}
\newtheorem{corollary}{Corollary}

\begin{document}

\title{Nash equilibrium with Sugeno payoff}
\author{Taras Radul}

\maketitle

Institute of Mathematics, Casimirus the Great University, Bydgoszcz, Poland;
\newline
Department of Mechanics and Mathematics, Lviv National University,
Universytetska st.,1, 79000 Lviv, Ukraine.
\newline
e-mail: tarasradul\@ yahoo.co.uk



\begin{abstract} This paper is devoted to  Nash equilibrium for games in capacities. Such games with payoff expressed by Choquet
integral were considered in \cite{KZ} and existence of Nash equilibrium was proved. We also consider games in capacities but with expected payoff expressed by Sugeno
integral. We prove existence of Nash equilibrium using categorical methods and abstract convexity theory.
 \end{abstract}

\section{Introduction}

The classical Nash equilibrium theory is based on fixed point theory and was developed in frames of linear convexity. The mixed strategies of a player are probability (additive) measures on a set of pure strategies. But an interest to Nash equilibria in more general frames is rapidly growing in last decades. There are also  results about  Nash equilibrium  for non-linear convexities. For instance, Briec and Horvath proved in  \cite{Ch} existence of Nash equilibrium point for $B$-convexity and MaxPlus convexity. Let us remark that MaxPlus convexity is
related to idempotent (Maslov) measures in the same sense as linear convexity is related to probability measures.

We can use additive measures only when we know precisely probabilities of all events considered in a game. However it is not the case
in many modern economic models. The decision theory under uncertainty considers a model when probabilities of states are either not known or imprecisely specified. Gilboa \cite{Gil} and Schmeidler  \cite{Sch} axiomatized  expectations expressed by Choquet
integrals attached to non-additive measures called capacities, as a formal approach to decision-making under uncertainty. Dow and Werlang \cite{DW} generalized this approach for two players game where belief of each player about a choice of the strategy by the other player is a capacity. This result was extended onto games with arbitrary finite number of players \cite{EK}.


Kozhan and Zaricznyi introduced in \cite{KZ} a formal mathematical generalization of Dow and Werlang's concept of Nash equilibrium of a game where players are allowed to form non-additive beliefs about opponent's decision but also to play their mixed non-additive strategies. Such game is called by authors game in capacities. The expected payoff function was there defined using a Choquet integral.  Kozhan and Zaricznyi proved existence theorem using a linear convexity on the space of capacities which is preserved by Choquet integral. There was stated a problem of existence of Nash equilibrium for another functors \cite{KZ}.

An alternative to so-called Choquet expected utility model is the qualitative decision theory.   The corresponding expected utility is expressed by Sugeno integral. See for example papers \cite{DP}, \cite{DP1}, \cite{CH1}, \cite{CH} and others.  Sugeno integral chooses a median value of utilities which is qualitative counterpart of the averaging operation by Choquet integral.

Following \cite{KZ} we introduce in this paper the general mathematical concept of Nash equilibrium of a game in capacities. However, motivated by the qualitative approach, we consider expected payoff function defined by Sugeno integral. To prove existence theorem for this concrete case, we consider more general framework which could unify all mentioned before situations and give us a method to prove theorems about existence of Nash equilibrium in different contexts. We use categorical methods and abstract convexity theory.

The notion of convexity considered in this paper is
considerably broader then the classic one; specifically, it is not
restricted to the context of linear spaces. Such convexities
appeared in the process of studying different structures like
partially ordered  sets, semilattices, lattices, superextensions
etc. We base our approach on the notion of topological convexity
from \cite{vV} where the general convexity theory is covered from axioms
to application in different areas. Particularly, there is proved Kakutani fixed point theorem for abstract convexity.

 Above mentioned constructions of the spaces of probability measures, idempotent measures and capacities are functorial and could be completed to monads (see \cite{RZ}, \cite{Z} and \cite{NZ} for more details). There was introduced  in \cite{R1} a convexity structure on
each $\F$-algebra  for any monad $\F$ in the category of compact Hausdorff spaces and continuous maps. Particularly, topological properties of monads with binary convexities were investigated.

We prove a counterpart of Nash theorem for an abstract convexity in this paper. Particularly, we consider binary convexities. These results we use to obtain Nash theorem for algebras of any L- monad with binary convexity. Since capacity monad is an L-monad with binary convexity \cite{R2}, we obtain as corollary the corresponding result for capacities.


\section{Games in capacities} By $\Comp$ we denote the category of compact Hausdorff
spaces (compacta) and continuous maps. For each compactum $X$ we denote by $C(X)$ the Banach space of all
continuous functions on $X$ with the usual $\sup$-norm. In what follows, all
spaces and maps are assumed to be in $\Comp$ except for $\R$ and
maps in sets $C(X)$ with $X$ compact Hausdorff.

We need the definition of capacity on a compactum $X$. We follow a terminology of \cite{NZ}.
A function $c$ which assign each closed subset $A$ of $X$ a real number $c(A)\in [0,1]$ is called an {\it upper-semicontinuous capacity} on $X$ if the three following properties hold for each closed subsets $F$ and $G$ of $X$:

1. $c(X)=1$, $c(\emptyset)=0$,

2. if $F\subset G$, then $c(F)\le c(G)$,

3. if $c(F)<a$, then there exists an open set $O\supset F$ such that $c(B)<a$ for each compactum $B\subset O$.

We extend a capacity $c$ to all open subsets $U\subset X$ by the formula $c(U)=\sup\{c(K)\mid K$ is a closed subset of $X$ such that $K\subset U\}$.

It was proved in \cite{NZ} that the space $MX$ of all upper-semicontinuous  capacities on a compactum $X$ is a compactum as well, if a topology on $MX$ is defined by a subbase that consists of all sets of the form $O_-(F,a)=\{c\in MX\mid c(F)<a\}$, where $F$ is a closed subset of $X$, $a\in [0,1]$, and $O_+(U,a)=\{c\in MX\mid c(U)>a\}$, where $U$ is an open subset of $X$, $a\in [0,1]$. Since all capacities we consider here are upper-semicontinuous, in the following we call elements of $MX$ simply capacities.

There is considered in \cite{KZ} a tensor product for capacities, which is a continuous map $\otimes:MX_1\times\dots\times MX_n\to M(X_1\times\dots\times X_n)$. Note that, despite the space of capacities contains the space of probability measures,  the tensor product of capacities does not extend tensor product of probability measures.

Due to Zhou \cite{Zh} we can identify the set $MX$ with some set of functionals defined on the space $C(X)$  using the
Choquet integral. We consider for each $\mu\in MX$ its value on a function $f\in  C(X)$ defined by the formulae
$$\mu(f)=\int fd\mu=\int_0^\infty\mu\{x\in X|f(X)\ge t\}dt+\int^0_{-\infty}(\mu\{x\in X|f(X)\ge t\}-1)dt$$

Let us remember the definition of Nash equilibrium. We consider a $n$-players game $f:X=\prod_{i=1}^n X_i\to\R^n$ with compact Hausdorff spaces of strategies $X_i$. The coordinate function $f_i:X\to \R$ we call payoff function of $i$-th player. For $x\in X$ and $t_i\in X_i$ we use the notation $(x;t_i)=(x_1,\dots,x_{i-1},t_i,x_{i+1},\dots,x_n)$. A point $x\in X$ is called a Nash equilibrium point if for each $i\in\{1,\dots,n\}$ and for each $t_i\in X_i$ we have $f_i(x;t_i)\le f_i(x)$. Kozhan and  Zarichnyj proved in \cite{KZ} existence
of Nash equilibrium for  game in capacities $ef:\prod_{i=1}^n MX_i\to\R^n$ with expected payoff functions defined by $$ef_i(\mu_1,\dots,\mu_n)=\int_{X_1\times\dots\times X_n}f_id(\mu_1\otimes\dots\otimes\mu_n)$$

Let us remark that the Choquet functional representation of capacities preserves the natural linear convexity structure on $MX$ which was used in the proof of existence of Nash equilibrium \cite{KZ}. However this representation does not preserve the capacity monad structure. (We will introduce the monad notion in Section 4).

    There was introduced \cite{R2} another functional representation of capacities using Sugeno integral (see also \cite{NR} for similar result). This representation preserves the capacity monad structure. Let us describe such representation. Fix any increasing homeomorphism $\psi:(0,1)\to\R$. We put additionally $\psi(0)=-\infty$, $\psi(1)=+\infty$ and assume $-\infty<t<+\infty$ for each $t\in\R$. We consider for each $\mu\in MX$ its value on a function $f\in  C(X)$ defined by the formulae
$$\mu(f)=\int_X^{Sug} fd\mu=\max\{t\in\R\mid \mu(f^{-1}([t,+\infty)))\ge\psi^{-1}(t)\}$$

    Let us remark that we use some modification of Sugeno integral. The original Sugeno integral \cite{Su} "ignores" function values outside the interval $[0,1]$ and we introduce a "correction" homeomorphism $\psi$ to avoid this problem. Now, following \cite{KZ}, we consider a game in capacities $sf:\prod_{i=1}^n MX_i\to\R^n$, but motivated by \cite{DP}, we consider Sugeno expected payoff functions defined by $$sf_i(\mu_1,\dots,\mu_n)=\int^{Sug}_{X_1\times\dots\times X_n}f_id(\mu_1\otimes\dots\otimes\mu_n)$$

    The main goal of this paper is to prove existence of Nash equilibrium for such game. Since Sugeno integral does not preserve linear convexity on $MX$ we can not use methods from \cite{KZ}. We will use some another natural convexity structure which has the binarity property (has Helly number 2). We will obtain some general result for such convexities which could be useful to investigate existence of Nash equilibrium for diverse construction. Finally, we will obtain the result for capacities as a corollary of  these general results.

\section{Binary convexities}

A family $\C$ of closed subsets of a compactum $X$ is
called a {\it convexity} on $X$ if $\C$ is stable for intersection
and contains $X$ and the empty set. Elements of $\C$ are called
$\C$-convex (or simply convex). Although we follow general concept  of abstract convexity from \cite{vV}, our definition is different.
We consider only closed convex sets. Such structure is called closure structure in \cite{vV}. The whole family of convex
sets in the sense of \cite{vV} could be obtained by the operation of
union of up-directed families. In what follows, we assume that each convexity contains all singletons.

A convexity $\C$ on $X$ is called $T_2$ if for each distinct $x_1$, $x_2\in
X$ there exist $S_1$, $S_2\in\C$ such that $S_1\cup S_2=X$,
$x_1\notin S_2$ and $x_2\notin S_1$. Let us remark  that if a convexity $\C$ on a compactum  $X$ is $T_2$,  then $\C$ is a  subbase for closed sets.
A convexity $\C$ on $X$ is called $T_4$ (normal) if for each disjoint $C_1$, $C_2\in
\C$ there exist $S_1$, $S_2\in\C$ such that $S_1\cup S_2=X$,
$C_1\cap S_2=\emptyset$ and $C_2\cap S_1=\emptyset$.

Let $(X,\C)$, $(Y,\D)$ be two
compacta with convexity structures. A continuous map $f:X\to Y$ is
called {\it CP-map} (convexity preserving map) if $f^{-1}(D)\in\C$
for each $D\in\D$; $f$ is
called {\it CC-map} (convex-to-convex map) if $f(C)\in\D$
for each $C\in\C$.

By a multimap (set-valued map) of a set $X$ into a set $Y$ we mean a map $F:X\to 2^Y$. We use the notation $F:X\multimap Y$. If $X$ and $Y$ are topological spaces, then a multimap $F:X\multimap Y$ is called upper semi-continuous (USC) provided for each open set $O\subset Y$ the set $\{x\in X\mid F(x)\subset O\}$ is open in $X$. It is well-known that a multimap is USC iff its graph is closed in $X\times Y$.

Let  $F:X\multimap X$ be a  multimap. We say that a point $x\in X$ is a fixed point of $F$ if $x\in F(x)$.
The following counterpart of Kakutani theorem for abstract convexity is a partial case of Theorem 3 from \cite{W} (it also could be obtain combining Theorem 6.15, Ch.IV and Theorem 4.10, Ch.III from \cite{vV}).

\begin{theorem}\label{KA} Let $\C$ be a normal convexity on a compactum $X$ such that all convex sets are connected and $F:X\multimap X$ is a USC multimap with values in $\C$. Then $F$ has a fixed point.
\end{theorem}

Let $\C$ be a family of subsets of a compactum $X$. We say that  $\C$ is {\it linked} if the intersection of every  two elements is non-empty. A convexity $\C$ is called {\it binary} if the intersection of every  linked subsystem of $\C$ is non-empty.

\begin{lemma}\label{BC} Let $\C$ be a $T_2$ binary convexity on a continuum $X$. Then $\C$ is normal and all convex sets are connected.
\end{lemma}

\begin{proof} The first assertion of the lemma is proved in Lemma 3.1  \cite{RZ}. Let us prove the second one. Consider any $A\in\C$. There was defined in \cite{MV} a  retraction $h_A:X\to A$ by the formula $h_A(x)=\cap\{C\in\C\mid x\in C$ and $C\cap A\ne\emptyset\}$. Hence $A$ is connected and the lemma is proved.
\end{proof}

Now we can reformulate Theorem  \ref{KA} for binary convexities.

\begin{theorem}\label{KB} Let $\C$ be a  $T_2$ binary convexity on a continuum $X$ and $F:X\multimap X$ is a USC multimap with values in $\C$. Then $F$ has a fixed point.
\end{theorem}

Now, let $\C_i$ be a convexity on $X_i$. We say that the function $f_i:X\to\R$ is quasi concave by $i$-th coordinate if we have $(f_i^x)^{-1}([t;+\infty))\in\C_i$ for each $t\in\R$ and $x\in X$ where $f_i^x:X_i\to\R$ is a function defined as follows $f_i^x(t_i)=f_i(x;t_i)$ for $t_i\in X_i$.

\begin{theorem}\label{NN} Let $f:X=\prod_{i=1}^n X_i\to\R^n$ be a game with a  normal convexity  $\C_i$ defined  on each compactum $X_i$ such that all convex sets are connected, the function $f$ is continuous  and the function $f_i:X\to\R$ is quasi concave by $i$-th coordinate for each $i\in\{1,\dots,n\}$. Then there exists a Nash equilibrium point.
\end{theorem}

\begin{proof} Fix any $x\in X$. For each $i\in\{1,\dots,n\}$ consider a set $M_i^x\subset X_i$ defined as follows $M_i^x=\{t\in X_i\mid f_i^x(t)=\max_{s\in X_i}f_i^x(s)\}$. We have that $M_i^x$ is a closed subset $X_i$. Since the function $f_i:X\to\R$ is quasi concave by $i$-th coordinate, we have that $M_i^x\in\C_i$. Define a multimap $F:X\multimap X$ by the formulae $F(x)=\prod_{i=1}^n M_i^x$ for $x\in X$.

Let us show that $F$ is USC. Consider any point $(x,y)\in X\times X$ such that $y\notin F(x)$. Then there exists $i\in\{1,\dots,n\}$ such that $f_i^x(y_i
)<\max_{s\in X_i}f_i^x(s)\}$. Hence we can choose $t_i\in X_i$ such that $f_i(x;y_i)<f_i(x;t_i)$. Since $f_i$ is continuous, there exists a neighborhood
$O_x$ of $x$ in $X$ and a neighborhood $O_{y_i}$ of $y_i$ in $Y_i$ such that for each $x'\in O_x$ and $y_i'\in O_{y_i}$ we have $f_i(x;y_i')<f_i(x;t_i)$. Put $O_y=(\pr_i)^{-1}(O_{y_i})$. Then for each $(x',y')\in O_x\times O_y$ we have $y'\notin F(x')$. Thus the graph of $F$ is closed in $X\times Y$, hence $F$ is upper semicontinuous.

We consider on $X$ the family $\C=\{\prod_{i=1}^n C_i\mid C_i\in\C_i\}$. It is easy to see that $\C$ forms a  normal convexity  on  compactum $X$ such that all convex sets are connected.
Then by Theorem \ref{KA} $F$ has a fixed point which is a Nash equilibrium point.
\end{proof}

Now, the following corollary follows from the previous theorem and Lemma \ref{BC}.

\begin{corollary}\label{NB} Let $f:X=\prod_{i=1}^n X_i\to\R^n$ be a game such that there is defined a $T_2$ binary convexity $\C_i$ on each continuum $X_i$, the function $f$ is continuous  and the function $f_i:X\to\R$ is quasi concave by $i$-th coordinate for each $i\in\{1,\dots,n\}$. Then there exists a Nash equilibrium point.
\end{corollary}

\section{L-monads and its algebras}

We apply Corollary \ref{NB} to study games defined on algebras of binary L-monads. We recall some categorical notions (see \cite{Mc} and \cite{TZ}
for more details). We define them only for the
category $\Comp$. Let $F:\Comp\to\Comp$ be a covariant functor. A functor $F$ is called  continuous if it preserves the limits of inverse
systems.
In what follows, all functors  assumed to preserve
monomorphisms, epimorphisms, weight of infinite compacta. We also assume that our functors are continuous.
For a  functor $F$ which preserves
monomorphisms and an embedding
$i:A\to X$ we shall identify the space $FA$ and the subspace
$F(i)(FA)\subset FX$.

 A {\it monad} $\T=(T,\eta,\mu)$ in the category
$\Comp$ consists of an endofunctor $T:{\Comp}\to{\Comp}$ and
natural transformations $\eta:\Id_{\Comp}\to T$ (unity),
$\mu:T^2\to T$ (multiplication) satisfying the relations $\mu\circ
T\eta=\mu\circ\eta T=${\bf 1}$_T$ and $\mu\circ\mu T=\mu\circ
T\mu$. (By $\Id_{\Comp}$ we denote the identity functor on the
category ${\Comp}$ and $T^2$ is the superposition $T\circ T$ of
$T$.)

Let $\T=(T,\eta,\mu)$ be a monad in the category ${\Comp}$. The
pair $(X,\xi)$ where $\xi:TX\to X$ is a map is called a $\T$-{\it
algebra} if $\xi\circ\eta X=id_X$ and $\xi\circ\mu X=\xi\circ
T\xi$. Let $(X,\xi)$, $(Y,\xi')$ be two $\T$-algebras. A map
$f:X\to Y$ is called a $\T$-algebras morphism if $\xi'\circ
Tf=f\circ\xi$.

Let $(X,\xi)$ be an $\F$-algebra for a monad $\F=(F,\eta,\mu)$ and
$A$ is a closed subset of $X$. Denote by $f_A$ the quotient map
$f_A:X\to X/A$ (the classes of equivalence are one-point sets $\{x\}$ for $x\in X\setminus A$ and the set $A$) and put $a=f_A(A)$. Denote $A^+=(Ff_A)^{-1}(\eta(X/A)(a))$.    Define the $\F$-{\it convex
hull} $C_\F(A)$ of $A$ as follows
$C_\F(A)=\xi(A^+)$. Put additionally
$C_\F(\emptyset)=\emptyset$. We define the family
$\C_\F(X,\xi)=\{A\subset X|A $ is closed and $\C_\F(A)=A\}$.
Elements of the family $\C_\F(X,\xi)$ we call $\F$-{\it convex}. It was shown in \cite{R1} that the family $\C_\F(X,\xi)$ forms a
convexity on $X$, moreover, each morphism of $\F$-algebras is a $CP$-map. Let us remark that one-point sets are always $\F$-convex.

We don't know if the convexities  we have introduced
 are $T_2$. We consider in this section a
class of monads generating convexities which have this
property. The class of $L$-monads was introduced in \cite{R1} and it contains many well-known monads in
$\Comp$ like superextension, hyperspace, probability measure, capacity, idempotent measure  etc.
 For $\phi\in C(X)$ by $\max\phi$ ($\min\phi$) we denote $\max_{x\in
X}\phi(x)$ ($\min_{x\in X}\phi(x)$) and $\pi_\phi$ or $\pi(\phi)$
denote the corresponding projection $\pi_\phi:\prod_{\psi\in
C(X)}[\min\psi,\max\psi]\to[\min\phi,\max\phi]$.  It was shown in
\cite{R3} that for each L-monad $\F=(F,\eta,\mu)$ we can consider $FX$
as subset of the product $\prod_{\phi\in C(X)}[\min\phi,\max\phi]$,
moreover, we have $\pi_\phi\circ \eta X=\phi$,  $\pi_\phi\circ \mu
X=\pi(\pi_\phi)$ for all $\phi\in C(X)$ and $\pi_\psi\circ
Ff=\pi_{\psi\circ f}$ for all $\psi\in C(Y)$, $f:X\to Y$. We could consider these properties of $L$-monads as a definition \cite{R3}.

We say that an L-monad $\F=(F,\eta,\mu)$ weakly
preserves preimages if for each map $f:X\to Y$ and each closed
subset $A\subset Y$ we have
$\pi_\phi(\nu)\in[\min\phi(f^{-1}(A)),$ $\max\phi(f^{-1}(A))]$ for
each $\nu\in (Ff)^{-1}(A)$ and $\phi\in C(X)$ \cite{R1}. It was shown in \cite{R1} that for each   L-monad $\F$ which weakly
preserves preimages the convexity $\C_\F(FX,\mu X)$ is $T_2$.

\begin{lemma}\label{CC} Let $(X,\xi)$ be an $\F$-algebra for an $L$-monad $\F=(F,\eta,\mu)$ which weakly preserves preimages. Then the map $\xi:FX\to X$ is a CC-map for convexities $\C_\F(FX,\mu)$ and $\C_\F(X,\xi)$ respectively.
\end{lemma}

\begin{proof} Consider any $B\in \C_\F(FX,\mu)$. We should show that $\xi(B)\in\C_\F(X,\xi)$. Denote by  $\chi:X\to X/\xi(B)$ the  quotient map and put  $b=\chi(\xi(B))$. Consider any $\A\in FX$ such that $F\chi(\A)=(\eta(X/\xi(B))(b))$. We should show that $\xi(\A)\in\xi(B)$.

Consider the quotient map  $\chi_1:FX\to FX/B$ and put $b_1=\chi_1(B)$. There exists a (unique) continuous map $\xi':FX/B\to X/\xi(B)$ such that $\xi'(b_1)=b$ and $\xi'\circ \chi_1=\chi\circ \xi$. Put $\D=F(\eta X)(\A)$. We have $F\xi(\D)=\A$, hence $F\xi'\circ F\chi_1(\D)=F\chi\circ F\xi(\D)=F\chi(\A)=\eta(X/\xi(B))(b)$. Since $F$ weakly preserves preimages, we have $F\chi_1(\D)=\eta(FX/B)(b_1)$. Since $B\in \C_\F(FX,\mu)$, we have $\mu X(\D)\in B$. Hence $\xi(\A)=\xi\circ F\xi(\D)=\xi\circ \mu(\D)\in\xi(B)$. The lemma is proved.
\end{proof}

We call a monad $\F$ binary if
$\C_\F(X,\xi)$ is binary for each $\F$-algebra $(X,\xi)$.

\begin{lemma}\label{BT} Let $\F=(F,\eta,\mu)$  be a binary L-monad  which weakly preserves preimages. Then for each $\F$-algebra $(X,\xi)$  the  convexity $\C_\F(X,\xi)$ is $T_2$.
\end{lemma}

\begin{proof} Consider any two distinct points $x$, $y\in X$. Since $\xi$ is a morphism of $\F$-algebras $(FX,\mu X)$ and $(X,\xi)$, it is a CP-map and we have $\xi^{-1}(x)$, $\xi^{-1}(y)\in \C_\F(FX,\mu)$. Since $\C_\F(FX,\mu)$ is $T_2$ and binary, it is normal by Lemma \ref{BC}. Hence we can choose $L_1$, $L_2\in \C_\F(FX,\mu)$ such that $L_1\cup L_2=FX$ and $L_1\cap\xi^{-1}(x)=\emptyset$, $L_2\cap\xi^{-1}(y)=\emptyset$. Then we have $\xi(L_1)$, $\xi(L_2)\in\C_\F(X,\xi)$ by Lemma \ref{CC}, $\xi(L_1)\cup\xi(L_2)=X$, $x\notin L_1$ and $y\notin L_2$. The lemma is proved.
\end{proof}

Consider any  L-monad $\F=(F,\eta,\mu)$. It is easy to check that for each segment $[a,b]\subset\R$ the pair $([a,b],\xi_{[a,b]})$ is an $F$-algebra where $\xi_{[a,b]}=\pi_{\id_{[a,b]}}$.
Consider a game $f:X=\prod_{i=1}^n X_i\to\R^n$ where for each compactum $X_i$ there exists a map $\xi_i:FX_i\to X_i$ such that the pair $(X_i,\xi_i)$ is an $\F$-algebra.  We say that the function $f_i:X\to\R$ is an $\F$-algebras morphism by $i$-th coordinate if for each $x\in X$ the function $f_i^x:X_i\to\R$ is a morphism of $\F$-algebras $(X_i,\xi_i)$ and $([\min f_i^x,\max f_i^x],\xi_{[\min f_i^x,\max f_i^x]})$.

\begin{theorem}\label{NA} Let $\F=(F,\eta,\mu)$  be a binary L-monad  which weakly preserves preimages. Let $f:X=\prod_{i=1}^n X_i\to\R^n$ be a game such that there is defined an $\F$-algebra map $\xi_i:FX_i\to X_i$ on each continuum $X_i$, the function $f$ is continuous  and the function $f_i:X\to\R$ is an $\F$-algebras morphism by $i$-th coordinate for each $i\in\{1,\dots,n\}$. Then there exists a Nash equilibrium point.
\end{theorem}

\begin{proof} Since for each $x\in X$ the function $f_i^x:X_i\to\R$ is an $\F$-algebras morphism, it is a CP-map, hence quasi concave. Now, our theorem follows from Lemma \ref{BT} and Corollary \ref{NB}.
\end{proof}

\section{Pure and mixed strategies}

Let $\F=(F,\eta,\mu)$  be a binary L-monad  which weakly preserves preimages.
We consider Nash equilibrium for free algebras $(FX,\mu X)$ in this section. Points of a compactum $X$ we call pure strategies and points of $FX$ we call mixed strategies. Such approach is a natural generalization of the model from  \cite{KZ} where spaces of capacities $MX$ were considered.

We consider a game $u:X=\prod_{i=1}^n X_i\to\R^n$ with compact Hausdorff spaces of pure strategies $X_1,\dots,X_n$ and continuous payoff functions $u_i:\prod_{i=1}^n X_i\to\R$.

It is well known how to construct the tensor product of two (or finite number) probability measures. This operation was generalized in \cite{TZ} for each monad in the category $\Comp$. More precisely there was constructed for each compacta $X_1,\dots,X_n$ a continuous map $\otimes:\prod_{i=1}^n F X_i\to F(\prod_{i=1}^n X_i)$ which is natural by each argument and for each $i$ we have $F(p_i)\circ\otimes= \pr_i$ where $p_i:\prod_{j=1}^nX_j\to X_i$ and $\pr_i:\prod_{j=1}^n FX_j\to FX_i$ are natural projections.

We define the payoff functions $eu_i:FX_1\times\dots\times FX_n\to\R$ by the formula  $eu_i=\pi_{u_i}\circ\otimes$. Evidently, $eu_i$ is continuous. Consider any $t\in\R$ and $\nu\in  FX_1\times\dots\times FX_n$. Then we have $(eu_i^\nu)^{-1}[t;+\infty)=\{\mu_i\in FX_i\mid eu_i(\nu;\mu_i)\ge t_i\}=l^{-1}(\pi_{u_i}^{-1}[t;+\infty)\cap\{\nu_i\}\times\dots\times FX_i\times\dots\times\{\nu_n\})$, where $l:FX_i\to\prod_{j=1}^n FX_j$ is an embedding defined by $l(\mu_i)=(\nu;\mu_i)$ for $\mu_i\in FX_i$. A structure of $\F$-algebra on the product $\prod_{j=1}^n FX_j$ of $\F$-algebras $(FX_i,\mu X_i)$ is given by a map $\xi:F(\prod_{i=1}^n FX_i)\to\prod_{i=1}^n FX_i$ defined by the formula $\xi=(\mu X_i\circ F(p_i))_{i=1}^n$. It is easy to check that a product of convex in $FX_i$ sets is convex in $\prod_{i=1}^n FX_i$. Since $\F$ weakly preserves preimages, $\pi_{u_i}^{-1}[t;+\infty)$ is convex in $\prod_{i=1}^n FX_i$. It is easy to see that $l$ is a CP-map, hence the map $eu_i$ is quasiconcave on $i$-th coordinate.

Hence, using Corollary \ref{NB}, we obtain the following theorem.

\begin{theorem} The game with payoff functions $eu_i$ has a Nash equilibrium point provided each $FX_i$ is connected.
\end{theorem}

Now, consider a game in capacities with Sugeno payoff functions introduced in the beginning of the paper.

The assignment $M$ extends to the capacity functor $M$ in the category of compacta, if the map $Mf:MX\to MY$ for a continuous map of compacta $f:X \to Y$ is defined by the formula $Mf(c)(F)=c(f^{-1}(F))$ where $c\in MX$ and $F$ is a closed subset of $X$. This functor was completed to the monad $\M=(M,\eta,\mu)$ \cite{NZ}, where the components of the  natural transformations are defined as follows: $\eta X(x)(F)=1$ if $x\in F$ and $\eta X(x)(F)=0$ if $x\notin F$;
$\mu X(\C)(F)=\sup\{t\in[0,1]\mid \C(\{c\in MX\mid c(F)\ge t\})\ge t\}$, where $x\in X$, $F$ is a closed subset of $X$ and $\C\in M^2(X)$.
Since capacity monad $\M$ is a binary L-monad  which weakly preserves preimages with $\pi_\varphi(\nu)=\int_X^{Sug} fd\nu$ for any $\nu\in MX$ and $\varphi\in C(X)$ \cite{R2}, we obtain as a consequence

\begin{corollary}\label{NC} A game in capacities $sf:\prod_{i=1}^n MX_i\to\R^n$ with Sugeno payoff functions has a Nash equilibrium point.
\end{corollary}




\end{document}